\documentclass[9pt]{jdg-p-1-2}

\newtheorem{theorem}{Theorem}[section]

\newtheorem{corollary}[theorem]{Corollary}
\newtheorem{proposition}[theorem]{Proposition}

\theoremstyle{definition}

\theoremstyle{remark}
\newtheorem{remark}[theorem]{Remark}

\def\Ric{\text{Ric}}

\def\R{\Bbb R}
\def\P{\Bbb P}
\def\Sph{\Bbb S}

\def\Rc{\operatorname{{\bf R}}}

\def\id{\operatorname{id}}
\def\Ric{\operatorname{Ric}}

\def\tr{\operatorname{tr}}

\def\tri{\operatorname{tri}}

\numberwithin{equation}{section}

\begin{document}

\title{ On $4$-dimensional gradient shrinking solitons}

\author{Lei Ni}
\address{Department of Mathematics, University of California at San Diego, La Jolla, CA 92093}

\email{lni@math.ucsd.edu}
\author{Nolan  Wallach}
\address{Department of Mathematics, University of California at San Diego, La Jolla, CA 92093}

\email{nwallach@math.ucsd.edu}

\date{September 2007}
\thanks{ The first author's research  was
supported in part by NSF grant DMS-0504792  and an Alfred P. Sloan
Fellowship, USA.
The second author's research was partially supported by an NSF Summer Grant.
}


\begin{abstract}In this paper we classify the four dimensional gradient shrinking solitons under certain curvature conditions satisfied by all solitons arising from finite time singularities of Ricci flow on compact four manifolds with positive isotropic curvature. As a corollary we generalize a result of Perelman on three dimensional gradient shrinking solitons  to dimension four.
\end{abstract}

\maketitle

\section{Introduction}

The goal of this paper is to generalize a result of Perelman on three dimensional gradient
shrinking solitons to dimension four. In his surgery paper Perelman proved the following statement
\cite{P2}:

\begin{theorem} \label{perelman1} Any $\kappa$-non-collapsed (for some $\kappa>0$) complete
 gradient shrinking
 soliton $M^3$ with bounded
positive sectional curvature must be compact.
\end{theorem}

Combining with Hamilton's convergence (or curvature pinching) result
\cite{H82} (see also \cite{Ivey}) one can conclude that $M^3$ must
be isometric to a quotient of $\Sph^3$. The reader can find more detailed proof of this result in \cite{CaZ, KL, MT} and Theorem 9.79 of \cite{CLN}. We refer to \cite{NW1} for
the discussion on the uses of such a result in the study of Ricci
flow, an alternate proof to the above result, basic frame work for the
high dimensional cases and a related result in high dimensions.

For four manifolds, in \cite{H86}, Hamilton proved that for any
compact Riemannian manifold with positive curvature operator, the
Ricci flow deforms it into a metric of constant curvature. Such a
result has been generalized by H. Chen \cite{Ch} to manifolds whose
curvature operator is  $2$-positive. (Recently, in a foundational
work B\"ohm and Wilking \cite{BW1}
 have generalized this  result to all dimensions.)
 However it is still unknown if there exists any four dimensional complete
 gradient shrinking solitons
 with positive curvature operator which is not compact.

  In \cite{H97}, Hamilton initiated another
 important direction, Ricci flow with surgery, and used the method to study the topology of
  four manifolds with positive isotropic curvature.

Recall from \cite{H86} that there is a natural splitting of $\wedge^2(\R^4)$ into self-dual and anti-self-dual parts and  one can write the curvature
operator $\Rc$ as
$$
\Rc=\left(\begin{matrix} A & B\cr B^t & C\end{matrix}\right)
$$
according to the decomposition $\wedge^2(\R^4)=\wedge_+ \oplus
\wedge_{-}$. We may choose the basis for $\wedge_+$ and $\wedge_{-}$
as
\begin{eqnarray*}
\varphi_1=\frac{1}{\sqrt{2}}\left(e_1\wedge e_2 +e_3\wedge
e_4\right), \quad &\quad& \quad
\psi_1=\frac{1}{\sqrt{2}}\left(e_1\wedge e_2 -e_3\wedge
e_4\right),\\
\varphi_2=\frac{1}{\sqrt{2}}\left(e_1\wedge e_3 +e_4\wedge
e_2\right), \quad &\quad& \quad
\psi_2=\frac{1}{\sqrt{2}}\left(e_1\wedge e_3 -e_4\wedge
e_2\right),\\
\varphi_3=\frac{1}{\sqrt{2}}\left(e_1\wedge e_4 +e_2\wedge
e_3\right),\quad &\quad&\quad
\psi_3=\frac{1}{\sqrt{2}}\left(e_1\wedge e_4 -e_2\wedge e_3\right),
\end{eqnarray*}
where $\{e_1, e_2, e_3, e_4\}$ is a positively oriented basis. The
first Bianchi identity implies that $\tr(A)=\tr(C)=\frac{S}{4}$
where $S$ is the scalar curvature. Note that $A$ and $C$ are
symmetric. Let $A_1\le A_2\le A_3$ and $C_1\le C_2\le C_3$ be
eigenvalues of $A$ and $C$ respectively. Then $\Rc$ has positive
isotropic curvature amounts to that $A_1+A_2>0$ and $C_1+C_2>0$.

In \cite{H97}, it was shown that on the blow-up limit of any finite
time singularity of Ricci flow on a compact $4$-manifold initially
with positive isotropic curvature, there exists $\delta >$ depending only on the initial manifold
such that the following pinching estimates hold
\begin{equation}\label{upic}
A_1 \ge \delta A_3,\quad C_1\ge \delta C_3,\quad  A_1C_1 \ge B_3^2
\end{equation}
where $0\le B_1\le B_2\le B_3$ are singular values of $B$. We say
that $\Rc$ has {\it uniformly positive isotropic curvature} if
(\ref{upic}) holds with $A_1C_1>0$. Note that this implies $\Rc\ge
0$. In view of the work \cite{H97} (see also related work
\cite{Chen-Zhu}) for the study of
 the Ricci flow on four manifolds with positive isotropic curvature it  is
useful to have a classification of gradient shrinking solitons with uniformly
positive isotropic curvature
 in the sense of (\ref{upic}).

On the other hand, in general on a gradient shrinking solitons with
 positive isotropic curvature,
it is not clear to the authors whether or not (\ref{upic}) always
holds. We say that a
 Riemannian four manifold $M$ has {\it weakly uniformly positive isotropic
  curvature}
 if there exists $\Sigma>0$ such that
\begin{equation}\label{wpic}
\left(\frac{B^2_3}{(A_1+A_2)(C_1+C_2)}\right)(x)\le \Sigma.
\end{equation}
By Theorem B2.1 of \cite{H97}, it is easy to infer that a gradient shrinking soliton with bounded
curvature satisfying (\ref{wpic}) must satisfy
\begin{equation}\label{wpic2}
\left(\frac{B^2_3}{(A_1+A_2)(C_1+C_2)}\right)(x)\le \frac{1}{4}.
\end{equation}

The main purpose of this article is to show a classification result on the gradient
shrinking solitons
satisfying a rather weak pinching condition:
\begin{equation}\label{pinch1}
\left(\frac{B^2_3}{(A_1+A_2)(C_1+C_2)}\right)(x)\le \exp(a(r(x)+1))
 \end{equation}
 for some $a>0$, where $r(x)$ is the distance function to a fixed point on the manifold.
 As in \cite{NW1}
  we also assume that the curvature tensor satisfies
 \begin{equation}\label{eq0}
 |R_{ijkl}|(x)\le \exp(b(r(x)+1))
 \end{equation}
for some $b>0$.

\begin{theorem}\label{main1} Any four dimensional complete gradient shrinking soliton with nonnegative
 curvature operator and  positive isotropic curvature satisfying (\ref{pinch1})
  and (\ref{eq0}) is
  either a quotient of $\Sph^4$ or a quotient of $\Sph^3\times \R$.
\end{theorem}

We should remark that in view of the examples  \cite{Ko, Co, FIK}
some conditions on the curvature operator are essential to obtain a
classification result as above. As a corollary of Theorem
\ref{main1} we have the following  four dimensional analogue
 of Theorem \ref {perelman1}.

\begin{corollary}
Any four dimensional gradient shrinking soliton with  positive
curvature operator satisfying (\ref{pinch1}) and (\ref{eq0})   must
be compact.
\end{corollary}

Note that there exists a general compactness result \cite{NW2} under
a certain pinching condition on the curvature operator, provided
that  the curvature operator is bounded. But the condition
(\ref{pinch1}) is a much weaker one since the curvature operator
pinching of \cite{NW2} implies that there exists $\epsilon>0$ with
$A_1\ge \epsilon A_3$, $C_1\ge \epsilon C_3$, which further implies
$(A_1+A_2)(C_1+C_2)\ge \epsilon' S^2\ge \epsilon'\delta$ for some
positive $\epsilon'$ and $\delta$ (by Proposition 1.1 of \cite{N}).
From the last estimate and the boundedness of curvature one can
deduce  (\ref{wpic}).

Combining the pinching result Theorem B1.1  of \cite{P1}, and Proposition 11.2 of \cite{P1} (see also \cite{N}), one can conclude that the {\it asymptotic soliton},  borrowing the terminology from \cite{P2},  arising from the singularity of Ricci flow on a four manifold with positive isotropic curvature, has nonnegative curvature operator, satisfies (\ref{upic}) and has at most quadratic curvature growth. Hence one can apply Theorem \ref{main1} to obtain a classification on  such {\it asymptotic solitons}.

\bigskip

{\it Acknowledgement}.  The first author would like to thank Christoph B\"ohm and Burkhard Wilking
for helpful discussions and suggestions.

\section{A preliminary result}

 From \cite{H86} we know that
$$
\Rc^{\#}=2\left(\begin{matrix} A^{\#} & B^{\#}\cr (B^t)^{\#} &
C^{\#}\end{matrix}\right),
$$
 the traceless part of $A$
and $C$ are $W_+$ and $W_{-}$, the self-dual part and the
anti-self-dual part of Weyl curvature,  and $B$ is the traceless
Ricci curvature. It is easy to see that $\tr(A)=\tr(C)=\frac{S}{4}$.
Here $S$ is the scalar curvature. Notice that $A^{\#}$, $B^{\#}$,
$C^{\#}$ are computed as transformations of $\wedge^2(\R^3)$. For
example $A^{\#}=\det(A)(A^t)^{-1}$, while
$B^{\#}=-\det(B)(B^{t})^{-1}$.

Let $\sigma^2=|\Ric|^2$ and $\tilde \sigma^2=|\Ric_0|^2$, where
$\Ric_0$ is the traceless part of Ricci tensor. Also let $\lambda_i$ be the eigenvalue of $\Ric_0$.
First we shall determine  how $\tilde \sigma^2$ is related to $B$. Direct
computation shows that
$$
B=\frac{1}{2}\left(\begin{matrix} R_{1212}-R_{3434} & R_{23}-R_{14}
& R_{24}+R_{13}\cr R_{23}+R_{14} & R_{1313}-R_{2424} &
R_{34}-R_{12}\cr R_{24}-R_{13}& R_{34}+R_{12} &
R_{1414}-R_{2323}\end{matrix}\right).
$$
Here $R_{ij}$ are the Ricci tensor components. From this we have the
following expression of $\Ric_0$ in terms of $B$:
\begin{equation}\label{B}
\Ric_0=\left(\begin{matrix} B_{11}+B_{22}+B_{33} & B_{32}-B_{23} &
B_{13}-B_{31} & B_{21}-B_{12}\cr B_{32}-B_{23} &
B_{11}-B_{22}-B_{33} & B_{21}+B_{12} & B_{13}+B_{31}\cr
B_{13}-B_{31}&B_{21}+B_{12}& B_{22}-B_{11}-B_{33} & B_{23}+B_{32}\cr
B_{21}-B_{12}&B_{13}+B_{31}& B_{23}+B_{32}&B_{33}-B_{11}-B_{22}
\end{matrix}\right).
\end{equation}
Direct computation shows that
$$
\tilde \sigma^2 = 4|B|^2 \quad \quad \mbox{ and } \quad \quad
\sum_1^4\lambda_j^3=-8 \tr(B^{\#}B^t).
$$
By \cite{NW1}, the classification result follows from the
non-positivity of $2\tri(\Rc)S-\sigma^2 |R_{ijkl}|^2$. Here
$\tri(\Rc)=2\langle \Rc^2+\Rc^{\#}, \Rc\rangle$. We now compute it
in terms of $A, B, C$. First it is easy to see that
\begin{equation}\label{q-1}
P\doteqdot=2\tri(\Rc)S-\sigma^2 |R_{ijkl}|^2=4\langle S (\Rc^2+\Rc^{\#})
-(\frac{S^2}{n}+\tilde \sigma^2)\Rc, \Rc \rangle.
\end{equation}
For the case $\dim(M)=4$ we have that
\begin{eqnarray*}
\langle \Rc^2+\Rc^{\#}, \Rc \rangle
&=&\tr(A^3)+\tr(C^3)+2\tr(A^{\#}A)+2\tr((B^t)^{\#}B)+2\tr(B^{\#}
B^t)\\&\quad&+2\tr(C^{\#}C) +3\tr(ABB^t)+3\tr(CB^t B)
\end{eqnarray*}
and
$$
\langle \Rc, \Rc\rangle =\tr(A^2)+\tr(C^2)+2|B|^2.
$$
Hence
\begin{eqnarray*}
\frac{1}{4}P&=&S\left(\tr(A^3)+\tr(C^3)+2\tr(A^{\#}A+C^{\#}C)+2\tr((B^t)^{\#}B)\right.\\
&\, &\left.+2\tr(B^{\#} B^t) +3\tr(ABB^t)+3\tr(CB^t
B)\right)\\
&\,&-(\frac{S^2}{4}+4|B|^2)(\tr(A^2)+\tr(C^2)+2|B|^2).
\end{eqnarray*}

Let $\overset{o} A$   be the traceless part of $A$. Similarly we
have $\overset{o}C$. By choosing suitable basis we may diagonalize
$\overset{o}A$ and $\overset{o}C$ such that we can  assume that
$$
A=\left(\begin{matrix} \frac{S}{12}+a_1 & 0&0\cr
0&\frac{S}{12}+a_2&0\cr 0& 0&
\frac{S}{12}+a_3\end{matrix}\right),\quad  \quad
C=\left(\begin{matrix} \frac{S}{12}+c_1 & 0&0\cr
0&\frac{S}{12}+c_2&0\cr 0& 0& \frac{S}{12}+c_3\end{matrix}\right).
$$
Now we can write
\begin{eqnarray}
\frac{1}{4}P&=&-S^2\left(\frac{1}{6}\sum_1^4\lambda_i^2+\sum_1^3 a_i^2 +\sum_1^3 c_i^2\right)
\nonumber\\
&\, & +4S \left(\sum_1^3 (a_i^3 +c_i^3)+6a_1a_2a_3 +6 c_1c_2c_3
-\frac{1}{2}\sum_1^4
\lambda_i^3\right)\\
&\,& +12S\left(a_1 b_1^2 +a_2 b_2^2 +a_3b_3^2 +c_1 \tilde b_1^2 +c_2
\tilde b_2^2 +c_3 \tilde b_3^2\right)\nonumber\\
 &\,&-2\left(\sum_1^4 \lambda_i^2\right)^2-4(\sum_{i=1}^4\lambda_i^2)
 \left(\sum_{i=1}^3 ( a_i^2+c_i^2)\right).\nonumber
\end{eqnarray}
Here $\sum_1^3 a_i=\sum_1^3 c_i =\sum_1^4\lambda_j =0$,
$b_i^2=\sum_{j=1}^3 B_{ij}^2$ and $\tilde b_i^2 =\sum_{j=1}^3
B_{ji}^2$. Hence $\sum_1^3 b_i^2 =\sum_1^3 \tilde b_i^2
=\frac{1}{4}\sum_1^4 \lambda_j^2$.

Check with some examples. After a scaling, we have that
$$
\Rc_{\Sph^4}=\left(\begin{matrix} \id & 0\cr 0 & \id\end{matrix}\right), \quad \Rc_{\Sph^3\times \R}
=\left(\begin{matrix} \id & F\cr F^t & \id\end{matrix}\right)\quad \Rc_{\Sph^2\times \Sph^2}
=\left(\begin{matrix} E & 0\cr 0 & E\end{matrix}\right), \quad \Rc_{\Sph^2\times \R^2}=
\left(\begin{matrix} E & E\cr E & E\end{matrix}\right)
$$
where
$$
F=\left(\begin{matrix} 1 & 0 &0\cr 0 & 1&0\cr 0 & 0& -1\end{matrix}\right)\quad
E=\left(\begin{matrix} 1 & 0 &0\cr 0 & 0&0\cr 0 & 0& 0\end{matrix}\right).
$$
It is easy to check that $P=0$ on the above examples. On the complex projective space
$\Rc_{\P^2}=\left(\begin{matrix} \id & 0\cr 0 & 3E\end{matrix}\right)$. $P=0$ in this case too!

 With suitable choices of
the orthornormal basis for $\Lambda_+$ and $\Lambda_{-}$ we can assume
that $A_i=\frac{S}{12}+a_i$, $C_i=\frac{S}{12}+c_i$. It is easy to see that
$\max\{\tilde b_i^2, b^2_i\} \le B_3^2$ for any $1\le i\le 3$.

   The main result of this section is to prove a special case of Theorem \ref{main1}.

     \begin{proposition}
Suppose that
     $BB^t=b^2\id$  for some $b$, $A$ and $C$ are positive semi-definite. Then
   $P\le 0$ and the universal cover of $M$ is either $\Sph^4$ or $\Sph^3\times \R$.

   \end{proposition}

   \begin{proof}
   Observing that
$$
\sum a_i^3+6a_1a_2a_3=3\sum a_i^3
$$
in order to show that $2\tri(\Rc)S-\sigma^2 |R_{ijkl}|^2\le 0$ it
suffices to show that
$$
-S^2 \sum a_i^2 +12 S \sum a_i^3-48b^2\sum a_i^2 \le 0.
$$
Here we have used that $\sum \lambda_i^2=12b^2$. Note that we have
the constraints that $\frac{S}{12}+a_i\ge 0$ and $\sum a_i =0$.
Using the fact that $\sum a_i^3 \le \frac{1}{\sqrt{6}}$ under the
constraints $\sum a_i=0$ and $\sum a_i^2 =1$, which can be obtained
by Proposition 4.1 of \cite{NW1}, we conclude that
$$
\frac{\sum a_i^3}{\sum a_i^2}\le \frac{1}{\sqrt{6}}a
$$
where $a^2=\sum a_i^2$. On the other hand, under the constraints
$\frac{S}{12}+a_i \ge 0$ and $\sum a_i =0$, the maximum of $\sum a_i^2$ is
$\frac{S^2}{24}$, which can be better seen by  expressing everything
in terms of $A_i=\frac{S}{12}+a_i\ge 0$. This shows that $-S^2 \sum
a_i^2 +12 S \sum a_i^3\le 0$ in view of $S> 0$. We can handle the
terms with $c_i$'s similarly. Furthermore, $-S^2 \sum a_i^2 +12 S
\sum a_i^3-48b^2\sum a_i^2=S^2 \sum c_i^2 +12 S \sum c_i^3-48b^2\sum
c_i^2=0$ implies either $b=0$ and $a_3=c_3=\frac{S}{4}$,
$a_1=a_2=c_1=c_2=-\frac{S}{12}$, or $a_i=c_i=0$, which is locally
conformally flat. The first case is excluded by the positivity of
the isotropic curvature. The second case was reduced to the previous
result of authors in \cite{NW1}. Evoking the proof of Corollary 4.2 of \cite{NW1} we obtain a complete classification
for this
special case. Note that we have used that $A$ and $C$ are semi-positive definite to
ensure that $\max\{\sum a_i^2, \sum c_i^2\}\le \frac{S^2}{24}$.
\end{proof}

In the next section we shall reduce the proof of Theorem \ref{main1} to this special case.

\begin{remark} It was pointed out to us by Christoph B\"ohm that the method of this
section alone is not enough to obtain the classification result for gradient shrinking solitons
 with positive curvature operator, unlike the three dimensional cases treated in \cite{NW1}.
\end{remark}

\section{The proof}

First we observe that some of the ordinary differential inequalities in \cite{H86} also hold as
 partial differential inequalities. We list the ones needed below.

\begin{proposition}\label{ham1} Let $(M, g(t))$ be a solution to Ricci flow. Let $A_i$, $B_i$ and $C_i$
 be the components of curvature operator as defined in the first section. Then with respect to the time
 dependent moving frame,
\begin{eqnarray*}
\left(\frac{\partial}{\partial t}-\Delta\right)(A_1+A_2)&\ge& A_1^2+A_2^2+2(A_1+A_2)A_3 +B_1^2+B_2^2,\\
\left(\frac{\partial}{\partial t}-\Delta\right)(C_1+C_2)&\ge & C_1^2+C_2^2+2(C_1+C_2)C_3 +B_1^2+B_2^2,\\
\left(\frac{\partial}{\partial t}-\Delta\right)B_3 &\le & A_3B_3+C_3B_3 +2B_1B_2.
\end{eqnarray*}
\end{proposition}
The differential inequality can be understood in the sense of distributions.
\begin{proof} The proof is essentially a repeat of the methods in \cite{H97}. Using a time dependent moving frame
we have that
$$
\left(\frac{\partial}{\partial t}-\Delta\right) \Rc =\Rc^2 +\Rc^{\#}.
$$
Fix a point $(x_0, t_0)$, choose a local frame so that $A$ and  $C$ are diagonal at $x_0$. Notice that
$\sum_{i, j=1}^2 A_{ij}g^{ij}\ge A_1+A_2$ and equality holds at $(x_0, t_0)$.
Hence at $(x_0, t_0)$ we have that
\begin{eqnarray*}
\left(\frac{\partial}{\partial t}-\Delta\right)\left(\sum_{i, j=1}^2 A_{ij}g^{ij}\right)&=&
 \sum_{i, j=1}^2 g^{ij} \left(A^2+BB^{t}+2A^{\#}\right)_{ij}\\
&\ge & A_1^2+A_2^2+2(A_1+A_2)A_3 +B_1^2+B_2^2.
\end{eqnarray*}
In the last line one uses the same line of argument as in \cite{H86}.
This shows the partial differential inequality in the sense of barrier. It then follows from the PDE theory,
in viewing of the concavity of $A_1+A_2$,  that the inequality also holds in the sense of distribution.
The other two inequalities can be shown similarly.
\end{proof}
Now we let $\psi_1=A_1+A_2$, $\psi_2=C_1+C_2$, $\varphi=B_3$. Our assumption on $M$ has positive isotropic
curvature implies that $\psi_1>0, \psi_2>0$. In the computations below we also assume $B_3>0$.
But it will be  clear later on that this is not necessary. Proposition \ref{ham1} implies
\begin{eqnarray*}
\left(\frac{\partial}{\partial t}-\Delta\right)\log \left(\frac{\varphi^2}{\psi_1\psi_2}\right)
&\le & 2|\nabla \log \varphi|^2 -|\nabla \log \psi_1|^2-|\nabla \log \psi_2|^2-\frac{2B_1(B_3-B_2)}{B_3}\\
&\, & -\frac{(A_1-B_1)^2+(A_2-B_2)^2+2A_2(B_2-B_1)}{A_1+A_2}\\
&\,&-\frac{(C_1-B_1)^2+(C_2-B_2)^2
+2C_2(B_2-B_1)}{C_1+C_2}.
\end{eqnarray*}
We denote the last three expressions as $-E$. It is clear that $-E\le 0$ with
 equality holds only if $A_1=C_1=B_1=B_2=A_2=C_2=B_3$. In particular we have that $B_1=B_2=B_3$,
  namely $BB^t=b^2\id$. Using the above partial differential inequality we have that
\begin{eqnarray*}
\left(\frac{\partial}{\partial t}-\Delta\right)\left(\frac{\varphi^2}{\psi_1\psi_2}\right)^2&\le &
 \left(\frac{\varphi^2}{\psi_1\psi_2}\right)^2\left(4|\nabla \log \varphi|^2
 -2|\nabla \log \psi_1|^2-2|\nabla \log \psi_2|^2-2E\right)\\
&\, & -4\left(\frac{\varphi^2}{\psi_1\psi_2}\right)^2|2\nabla \log \varphi -\nabla \log \psi_1-
\nabla \log \psi_2|^2.
\end{eqnarray*}
Now we compute the gradient terms.
\begin{eqnarray*}
&\,&4|\nabla \log \varphi|^2-2|\nabla \log \psi_1|^2-2|\nabla \log \psi_2|^2-4|2\nabla \log \varphi
-\nabla \log \psi_1-\nabla \log \psi_2|^2\\
&=&-2|2\nabla \log \varphi -\nabla \log \psi_1-\nabla \log \psi_2|^2
+2\langle \nabla \log \frac{\varphi}{\psi_1}, \nabla \log (\varphi \psi_1)\rangle\\
&\,& +2\langle \nabla \log \frac{\varphi}{\psi_2}, \nabla \log (\varphi \psi_2)\rangle
-2\langle \nabla \log \frac{\varphi}{\psi_1}, \nabla \log \frac{\varphi}{\psi_1}\rangle\\
&\, & -2\langle \nabla \log \frac{\varphi}{\psi_2}, \nabla \log \frac{\varphi}{\psi_2}\rangle
-4\langle \nabla \log \frac{\varphi}{\psi_1}, \nabla \log \frac{\varphi}{\psi_2}\rangle\\
&=& -2|\nabla \log \frac{\varphi}{\psi_1} +\nabla \log \frac{\varphi}{\psi_2}|^2+
8\langle \nabla \log \varphi, \nabla \log \psi_1\rangle
+8\langle \nabla \log \varphi, \nabla \log \psi_2\rangle\\
&\,& -4|\nabla \log \psi_1|^2-4|\nabla \log \psi_2|^2-4|\nabla \log \varphi|^2-
4\langle \nabla \log \psi_1, \nabla \log \psi_2\rangle\\
&=& -2|\nabla \log \frac{\varphi}{\psi_1} +\nabla \log \frac{\varphi}{\psi_2}|^2
-2|\nabla \log \frac{\varphi}{\psi_1}|^2-2|\nabla \log \frac{\varphi}{\psi_2}|^2\\
&\,& +2\langle \nabla \log \frac{\varphi^2}{\psi_1\psi_2}, \nabla
(\log \psi_1\psi_2)\rangle.
\end{eqnarray*}
Putting together we have that
\begin{eqnarray}
\left(\frac{\partial}{\partial t}-\Delta\right)\left(\frac{\varphi^2}{\psi_1\psi_2}\right)^2&\le &
 -2\left(\frac{\varphi^2}{\psi_1\psi_2}\right)^2E
 +\langle \nabla \left(\frac{\varphi^2}{\psi_1\psi_2}\right)^2,
  \nabla (\log \psi_1\psi_2)\rangle \nonumber\\
&\, & -2\left(\frac{\varphi^2}{\psi_1\psi_2}\right)^2\left(|\nabla \log \frac{\varphi}{\psi_1}
+\nabla \log \frac{\varphi}{\psi_2}|^2\right. \label{key1}\\
&\,&\left.+|\nabla \log \frac{\varphi}{\psi_1}|^2+|\nabla \log \frac{\varphi}{\psi_2}|^2\right).\nonumber
\end{eqnarray}
It is clear that the right hand side of the above inequality can be rewritten so that
$\varphi>0$ is not really required.
Since $(M, g)$ is a gradient shrinking soliton, letting $f$ be the potential function,
the computation in the Section 1 of \cite{NW1} implies that
$$
\frac{\partial }{\partial t} \left(\frac{\varphi^2}{\psi_1\psi_2}\right)^2
=\langle \nabla f, \nabla \left(\frac{\varphi^2}{\psi_1\psi_2}\right)^2\rangle.
$$
Now  multiply both sides of (\ref{key1}) by $e^{-f+\log (\psi_1+\psi_2)}$   and integrate over the manifold:
\begin{eqnarray*}&\, &
\int_M \langle \nabla f, \nabla  \left(\frac{\varphi^2}{\psi_1\psi_2}\right)^2\rangle
e^{-f+\log (\psi_1\psi_2)}
-\int_M \left(\Delta \left(\frac{\varphi^2}{\psi_1\psi_2}\right)^2\right)e^{-f+\log (\psi_1\psi_2)}\\
&\le & -2\int_M \left(\frac{\varphi^2}{\psi_1\psi_2}\right)^2E e^{-f+\log (\psi_1\psi_2)}
+\int_M \langle \nabla \left(\frac{\varphi^2}{\psi_1\psi_2}\right)^2,
\nabla (\log \psi_1\psi_2)\rangle e^{-f+\log (\psi_1\psi_2)}\\
&\,& -2\int_M \left(\frac{\varphi^2}{\psi_1\psi_2}\right)^2\left(|\nabla \log \frac{\varphi}{\psi_1}
 +\nabla \log \frac{\varphi}{\psi_2}|^2\right)e^{-f+\log (\psi_1\psi_2)}\\
&\,& -2\int_M \left(\frac{\varphi^2}{\psi_1\psi_2}\right)^2
\left(|\nabla \log \frac{\varphi}{\psi_1}|^2+
|\nabla \log \frac{\varphi}{\psi_2}|^2\right)e^{-f+\log (\psi_1\psi_2)}.
\end{eqnarray*}
All the integrals are finite by the derivative estimates of Shi \cite{Sh1},
 the assumption (\ref{pinch1}) and (\ref{eq0}),  and Lemma 1.3 of \cite{NW1}
 asserting that $f(x)\ge \frac{1}{8}r^2(x)-C$ (with $C>0$ and $r(x)$ being distance function to
  a fixed point). As in \cite{NW1}, integration by parts can be performed on the term involving the
   Laplacian operator in the left hand side of the above inequality. After the integration by parts
   and some cancelations we have that
\begin{eqnarray*}
0&\le& -\int_M  \left(\frac{\varphi^2}{\psi_1\psi_2}\right)^2E e^{-f+\log (\psi_1\psi_2)}\\
&\, &
 -2\int_M \left(\frac{\varphi^2}{\psi_1\psi_2}\right)^2\left(|\nabla \log \frac{\varphi}{\psi_1}
 +\nabla \log \frac{\varphi}{\psi_2}|^2\right)e^{-f+\log (\psi_1\psi_2)}\\
&\,& -2\int_M \left(\frac{\varphi^2}{\psi_1\psi_2}\right)^2
\left(|\nabla \log \frac{\varphi}{\psi_1}|^2+|\nabla \log \frac{\varphi}{\psi_2}|^2\right)
e^{-f+\log (\psi_1\psi_2)},
\end{eqnarray*}
which  implies that
$$
E=|\nabla \log \frac{\varphi}{\psi_1} +\nabla \log \frac{\varphi}{\psi_2}|
=|\nabla \log \frac{\varphi}{\psi_1}|=|\nabla \log \frac{\varphi}{\psi_2}|=0.
$$
In particular, we conclude that $B B^t=b^2\id$.
\bibliographystyle{amsalpha}

\end{document}